\documentclass{amsart}
\usepackage{amsmath, amsthm, amscd, amssymb}
\usepackage{tikz}
\numberwithin{equation}{section}

\theoremstyle{plain}
\newtheorem{theorem}{Theorem}[section]

\newtheorem{lemma}[theorem]{Lemma}
\newtheorem{corollary}[theorem]{Corollary}
\numberwithin{equation}{section}

\theoremstyle{definition}

\newtheorem{remark}[theorem]{Remark}

\newtheorem{example}[theorem]{Example}

\DeclareMathOperator{\rank}{rank}

\def\Z{\mathbb{Z}}
\def\Q{\mathbb{Q}}
\def\R{\mathbb{R}}
\def\C{\mathbb{C}}
\def\cF{\cF}
\def\e{\mathbf{e}}

\def\cF{\mathcal{F}}

\def\St{\mathrm{St }}
\def\Lk{\mathrm{Lk }}
\renewcommand{\P}{\mathbb{P}}

\def\k{\mathbf{k}}

\tikzstyle{v}=[circle, draw, solid, fill=black!50, inner sep=0pt, minimum width=4pt]
\tikzstyle{v2}=[circle, draw, solid, fill=black, inner sep=0pt, minimum width=4pt]

\begin{document}
\title[Real toric varieties {of Type $B$}]{The Betti numbers of real toric varieties associated to Weyl chambers of type $B$}
%\author{Soojin Cho}
%\address{Department of mathematics, Ajou University,
%206, World cup-ro, Yeongtong-gu, Suwon, 443-749,
%Republic of Korea}
%\email{chosj@ajou.ac.kr}

\author{Suyoung Choi}
\address{Department of mathematics, Ajou University, 206, World cup-ro, Yeongtong-gu, Suwon 16499,  Republic of Korea}
\email{schoi@ajou.ac.kr}

\author{Boram Park}
\address{Department of mathematics, Ajou University, 206, World cup-ro, Yeongtong-gu, Suwon 16499, Republic of Korea}
\email{borampark@ajou.ac.kr}

\author[Hanchul Park]{Hanchul Park}
\address{School of Mathematics, Korea Institute for Advanced Study (KIAS), 85 Hoegiro Dongdaemun-gu, Seoul 02455, Republic of Korea}
\email{hpark@kias.re.kr}

\date{\today}
\subjclass[2010]{14M25, 57N65, 17B22, 52B22, 05A15}

%14M25   	Toric varieties, Newton polyhedra
%57N65   	Algebraic topology of manifolds
%17B22   	Root systems
%52B22   	Shellability
%05A15   	Exact enumeration problems, generating functions

\keywords{real toric variety, real toric manifold, Betti number, torsion-free cohomology, root system, Weyl chambers, Type $B$, generalized Euler number, Springer number, shellability}

\maketitle
\begin{abstract}
We compute the (rational) Betti number of real toric varieties associated to Weyl chambers of type $B$. Furthermore, we show that their integral cohomology is $p$-torsion free for all odd primes $p$.
\end{abstract}
\section{Introduction}
 A \emph{toric variety} of complex dimension $n$ is a normal algebraic variety over $\C$ with an effective algebraic action of $(\C \setminus \{O\})^n$ having an open dense orbit. A compact smooth toric variety is called a \emph{toric manifold}.  One of the most important facts on toric geometry is that there is a 1-1 correspondence between the class of toric varieties of complex dimension $n$ and the class of fans in $\R^n$. This fact is called the \emph{fundamental theorem of toric geometry}. In particular, a toric manifold $X$ of complex dimension $n$ {corresponds to} a complete regular fan $\Sigma_X$ in $\R^n$.

    Among toric manifolds, the class of toric manifolds associated to Weyl chambers has been considered since it is introduced by Procesi \cite{Procesi1990}.
    A classical construction associates to each root system a toric manifold whose fan corresponds to the reflecting hyperplanes of the root system and its weight lattice.
    It is natural to ask about the topology of the corresponding toric manifold.
    Note that the integral cohomology of a toric manifold is well-established by Jurkiwicz \cite{Jurkiewicz1980} for the projective cases and by Danilov \cite{Danilov1978} for general cases.
    For a {coefficient field} $\k$, the \emph{$i$th $\k$-Betti number} of a topological space $X$ is the rank of $H^i(X;\k)$ over $\k$, and it is denoted by $\beta^{i}(X;\k)$.
    One remarkable fact is that the Betti numbers of a toric manifold $X$ depend only on the face numbers of its associated fan $\Sigma_X$.
    Especially, the structures of the cohomology of toric manifolds associated to Weyl chambers have been studied by \cite{Procesi1990},  \cite{Stembridge1994}, \cite{Dolgachev-Lunts1994} and \cite{Abe2015}.

    On the other hand, the subset consisting of points with real coordinates of a toric manifold is called a \emph{real toric manifold}.
    Unlike toric manifolds, little is known about the topology of real toric manifolds.
    Let $X$ be a toric manifold and $X^\R$ its real toric manifold.
    By Davis and Januszkiewicz \cite{Davis-Januszkiewicz1991}, the $i$th $\Z_2$-Betti number of $X^\R$ is equal to the $2i$th $\Z$-Betti number of $X$, and, hence, it depends only on the face numbers.
    However, the Betti numbers with rational coefficients are not only determined by the face numbers.
    For instance, both the torus and the Klein bottle are real toric manifolds and their corresponding fans have the face structure combinatorially equivalent to the $4$-gon.
    Hence, their $\Z_2$-Betti numbers are the same while their $\Q$-Betti numbers are different.
    From this sense, the computation of the rational Betti number of real toric manifolds is difficult, and only a few examples have been computed so far.
    One known example is the real toric manifolds associated to Weyl chambers of type $A_n$ due to Henderson \cite{Henderson2012}. Interestingly, {their rational Betti numbers are} the Euler zigzag numbers. Arnol$'$d \cite{Arnold1992} has defined the notion of snake numbers as a generalization of the Euler zigzag numbers as follows: a \emph{snake of type $A_n$} (respectively $B_n$), or an \emph{$A_n$-snake} (respectively, \emph{$B_n$-snake}), is a sequence of integers $x_i$ satisfying the conditions:
      \begin{itemize}
      \item for $A_n$ : $x_0 < x_1 > x_2 < \cdots x_n$, $x_i \neq x_j$ {for $i\ne j$};
      \item for $B_n$ : $0< x_1 > x_2 < \cdots x_n$, $x_i \neq \pm x_j$ {for $i\ne j$};
%      \item for $D_n$ : $-x_2 <x_1 <x_2 >x_3 < \cdots x_n$, $x_i \neq \pm x_j$;
      \end{itemize}
      where $0 \leq x_i  \leq n$ {for all $i$} for $A_n$, and $1 \leq |x_i | \leq n$ for all $i$ for $B_n$.
%, $0 \leq |x_i| < n$ for $D_n$.
    Denote by $a_n$ (respectively, $b_n$) the number of $A_n$-snakes (respectively, $B_n$-snakes).
    The number $a_n$ is also known as the Euler zigzag number (see A000111 of \cite{oeis}), and the number $b_n$ is also known as the generalized Euler number or the Springer number (see A001586 of \cite{oeis}).

\begin{table}[h]
  \centering
    \begin{tabular}{|c|ccccccccccc}
      \hline
      $n$ & 0 & 1 & 2 & 3 & 4 & 5 & 6 & 7 & 8 & 9 &$\cdots$ \\ \hline
      $a_n$ & 1 & 1 & 1 & 2 & 5 & 16 & 61 & 272 & 1385 & 7936& $\cdots$ \\
      $b_n$ & 1 & 1 & 3 & 11 & 57 & 361 & 2763 & 24611 & 250737 & 2873041& $\cdots$ \\
      \hline
    \end{tabular}
  \caption{The list of $a_n$ and $b_n$ for small $n$ }
\end{table}

    The formula of the rational Betti number by Henderson was recovered by Suciu \cite{Suciu2012} later using the general formula for rational Betti numbers of real toric manifolds established by Suciu and Trevisan \cite{ST2012}.
 \begin{theorem}[\cite{Henderson2012}, \cite{Suciu2012}] \label{thm:A_n}
    Denote by $X^\R_{A_n}$ the real toric manifold associated to {the} Weyl chambers of type $A_n$.
     The $k$th $\Q$-Betti number of $X^\R_{A_n}$ is
     $$
        \beta^k (X^\R_{A_n};\Q) = \binom{n+1}{2k}a_{2k}.
     $$
 \end{theorem}

    We note that Choi and Park \cite{CP2} showed that the formula used in \cite{ST2012} works for not only $\Q$ coefficient but also arbitrary field $\k$ coefficient whose characteristic is not equal to $2$. Combining it with \cite{Suciu2012}, we obtain the following.
 \begin{corollary} \label{cor:A_n}
    The integral cohomology of $X^\R_{A_n}$ is $p$-torsion free for all odd primes $p$.
 \end{corollary}

 In this paper, we compute the rational Betti number of real toric manifolds associated to the Weyl chambers of type $B_n$, and show that their integral cohomologies are $p$-torsion free for all odd primes $p$. We prove the following:
    \begin{theorem}
        Denote by $X^\R_{B_n}$ the real toric manifold associated to {the Weyl chambers} of type $B_n$. Then, we have the following:
%        \begin{enumerate}
          \[
            \beta^k (X^\R_{B_n}; \Q) = \binom{n}{2k}b_{2k} + \binom{n}{2k-1}b_{2k-1}.
          \]
        Furthermore, their integral cohomologies are $p$-torsion free for all odd primes $p$.
    \end{theorem}
{It is worthwhile to note that the same techniques to prove the above theorem do not directly apply to the case of type $C$ and $D$, the other regular types. This is because the analogues for the shellability results like Lemma~\ref{lem:B:shellable} fail for type $C$ or $D$, making it hard to compute homology of the corresponding posets.}

This paper is organized as follows. In Section~2, we introduce preliminary facts including the formula of Suciu-Trevisan to compute the Betti numbers of real toric manifolds and the way to define projective toric manifolds associated to Weyl chambers. In Section~3, we prove the main theorem, that is, we compute the Betti numbers of real toric manifolds of type $B_n$.

\section{Preliminaries}
\subsection{The Betti numbers of real toric manifolds}
In this subsection, we shall introduce a formula of the Betti numbers of real toric manifolds.
    From now on, we restrict our interests in the projective toric manifolds and its real toric manifolds. Let $X$ be a projective toric manifold of complex dimension $n$ and $X^\R$ its real toric manifold. We assume that the associated fan $\Sigma_X$ of $X$ has $m$ rays $r_1, \ldots, r_m$. Then, $\Sigma_X$ can be regarded as a pair of an $(n-1)$-dimensional polytopal simplicial sphere $K$ with the vertex set $[m]=\{1,\ldots, m\}$ and a map $\lambda \colon [m] \to \Z^n$ such that
\begin{itemize}
  \item $\sigma = \{i_1, \ldots, i_\ell \} \in K$ if and only if $\{r_{i_1}, \ldots, r_{i_\ell}\}$ forms a cone in $\Sigma_X$, and
  \item $\lambda(i)$ is the primitive vector in the direction of $r_i$.
\end{itemize}
    We call $K$ the \emph{underlying simplical complex} of $X$ and $\lambda$ the \emph{characteristic map} of $X$. Furthermore, since $X$ is projective, there is a convex simple polytope $P$ with $m$ facets $F_1, \ldots, F_m$ whose face structure is isomorphic to $K$ and the outward normal vector of $F_i$ is $\lambda(i)$ for $i=1, \ldots, m$.

    Similarly to the fundamental theorem for toric geometry, it is known that as a $\Z_2$-space, a real toric manifold $X^\R$ is determined by {the pair $(K,\lambda^\R)$, where $\lambda^\R$ the composition map of $\lambda$ and the canonical quotient map $\Z \to \Z/2\Z$,} i.e., $\lambda^\R \colon [m] \stackrel{\lambda}{\to} \Z \to \Z/2\Z$. We call $\lambda^\R$ the \emph{$\Z_2$-characteristic map}, and we note that $\lambda^\R$ can be represented as  a $\Z_2$-matrix of size $n\times m$, called the \emph{$\Z_2$-characteristic matrix}.
    For each subset $S$ of $\{1,\ldots,n\}$, write $\lambda_S=\sum_{i\in S}\lambda_i$, where $\lambda_i$ is the $i$th row of $\lambda^\R$.
    Let $[m]_S :=\{ j \in [m] \mid  \text{the $j$th entry of $\lambda_S$ is nonzero}\} \subset [m]$.
    For such $S$ we define $K_S:= \{ \sigma \in K | \sigma \subset [m]_S\}$, and, as dual, $\displaystyle P_S:= \bigcup_{j\in [m]_S} F_j$.
    We note that the topological realization of $K_S$ is homotopy equivalent to $P_S$. Throughout this paper, we denote by $K$ the topological
realization of a simplicial complex $K$ if there is no danger of confusion.

    \begin{theorem}\label{formula}\cite{ST2012, CP2}
        Let $X$ be a toric manifold and $X^\R$ its real toric manifold. Let $\k$ be a ring where $2$ is invertible in $\k$. Then the $i$th Betti number $\beta^i (X^\R;\k)$ of $X^\R$ with coefficient $\k$ is given by
        $$
        \beta^i (X^\R;\k) = \sum_{S\subseteq [n]} \rank_{\k} \tilde H^{i-1}(P_S;\k) = \sum_{S\subseteq [n]} \rank_{\k} \tilde H^{i-1}(K_S;\k).
        $$
    \end{theorem}

Suciu and Trevisan in their unpublished paper \cite{ST2012} have established the formula for the rational Betti numbers of real toric manifolds.
Later, Choi and Park in \cite{CP2} have also derived a cohomology formula of real toric manifolds with the coefficient ring $G$, where $2$ is invertible in $G$. It should be noted that if the reduced cohomology of $P_S$ is $p$-torsion free for all $S \subseteq [n]$ and all odd primes $p$, then so is the cohomology of $X^\R$.
 This formula also determines a stable homotopy decomposition of a wider class of spaces called real toric spaces, see \cite{CKT}.

\subsection{Projective toric manifold associated to Weyl chambers}
As mentioned in Introduction, we mainly deal with the class of (real) toric manifolds associated to the decomposition given by Weyl chambers. Let $V$ be a finite dimensional real Euclidean space, $\Phi \subset V$ a root system, and $W$ its Weyl group. In $V$ we have the lattice $\Lambda = \{ v \in V \mid (v, \alpha) \in \Z \text{ for any $\alpha \in \Phi$} \}$ which defines an integral structure, where $(-,-)$ is the natural inner product. For each set $\Delta$ of simple roots in $\Phi$ we consider the cone $C_\Delta = \{ v \in V \mid (v, \alpha) >0 \text{ for any $\alpha \in \Delta$}\}$. These cones provide the rational polyhedral decomposition of $V$, i.e., the set of cones is a fan in $V$. Hence, it defines a projective toric variety, which is in fact smooth.

 From now on, let us consider the Weyl groups of regular types.
 Throughout this paper, for the Weyl group of type $A_n$, the corresponding toric variety, its fan, the underlying simplicial complex, the characteristic map, the corresponding real toric variety, and the $\Z_2$-characteristic map are denoted by $X_{A_n}$, $\Sigma_{A_n}$, $K_{A_n}$, $\lambda_{A_n}$, $X^\R_{A_n}$ and $\lambda^\R_{A_n}$, respectively. For the Weyl group of type $B_n$, the corresponding notions are similarly denoted by $X_{B_n}$, $\Sigma_{B_n}$, $K_{B_n}$, $\lambda_{B_n}$, $X^\R_{B_n}$ and $\lambda^\R_{B_n}$, respectively.

\subsection{Type $A_n$}
%    \blue{From Hanchul: Suyoung wants to eliminate this section if possible, since it deals with an already known fact. Even though I am not 100\% sure, it seems that a specific shelling of $\tilde{S}_{A_n}^{odd}$ is needed in the proof for Lemma~\ref{lem:B:shellable}, letting Lemma~\ref{lem:shellable-A-type} essential. I think Boram can take care of this best.}\orange{From Boram: I deleted the proof and added some definition in the footnote. And I changed this section to a subsection of Preliminaries.}

    In this subsection, we shall review a sketch of proof of Theorem~\ref{thm:A_n} and Corollary~\ref{cor:A_n}.
    The proof presented here is essentially the same with that by \cite{Suciu2012} or \cite{Choi-Park2015} for the special case when the corresponding graph is a complete graph.
    However, we enclose this subsection for the sake of self-contained readability. {Furthermore, it introduces a lemma needed to prove the main result.}

    Let $W$ be the Weyl group of type $A_n$.
    It is well-known that the vertices of $K_{A_n}$ can be identified by the nonempty proper subsets $I$ of $[n+1]$ and each $(\ell-1)$-dimensional simplex of $K_{A_n}$ is related to a nested $\ell$ nonempty proper subsets of $[n+1]$, that is,
    $\{I_{i_1}, \ldots, I_{i_\ell}\} \in K_{A_n}$ if and only if there is a permutation $\sigma$ on $[\ell]$ such that $I_{i_{\sigma(1)}} \subset \cdots \subset I_{i_{\sigma(\ell)}}$.
    In addition, the characteristic map $\lambda_{A_n}$ is
    $$
        \lambda_{A_n}(I) = \left\{
                             \begin{array}{ll}
                               \sum_{ k \in I } \varepsilon_{k}, & \hbox{if $\{n+1\} \not\in I$;} \\
                               \sum_{ k \in I \setminus \{n+1\} } \varepsilon_{k} - \varepsilon_{1}- \cdots - \varepsilon_{n}, & \hbox{if $\{n+1\} \in I$,}
                             \end{array}
                            \right.
    $$ where $\varepsilon_i$ is the $i$th standard vector of $\Z^n$.
    As consequence,
    $$
        \lambda^\R_{A_n}(I) = \left\{
                             \begin{array}{ll}
                               \sum_{ k \in I } \e_{k}, & \hbox{if $\{n+1\} \not\in I$;} \\
                               \sum_{ k \not\in I } \e_{k}, & \hbox{if $\{n+1\} \in I$,}
                             \end{array}
                              \right.
    $$ where $\e_i$ is the $i$th standard vector of $\Z_2^n$.

    From now on, let us compute the $\Q$-Betti number of $X^\R_{A_n}$.
    By Theorem~\ref{formula}, we have to consider $(K_{A_n})_S$ for all subsets $S \subset [n]$.
    Here are three important nontrivial steps.

    For an odd number $r$, define $K_{A_r}^{odd}$ as
    $$
        K_{A_r}^{odd} := (K_{A_r})_{[r]}.
    $$
    \begin{enumerate}
      \item $K_{A_r}^{odd}$ is homotopy equivalent to the wedge of spheres of dimension $\frac{r-1}{2}$.
      \item The reduced {Euler characteristic} $\tilde{\chi}(K_{A_r}^{odd})$ is $(-1)^{\frac{r-1}{2}} a_{r+1}$.
      \item For $S \subset [n]$ with $|S|=r$ or $|S|=r+1$ for some odd number $r$, $(K_{A_n})_S$ is homotopy equivalent to $K_{A_r}^{odd}$.
    \end{enumerate}

    By (1)--(3) together with Theorem~\ref{formula}, both Theorem~\ref{thm:A_n} and Corollary~\ref{cor:A_n} are immediately proved.

\section{Type $B_n$}

  %  In order to discuss the toric manifolds associated to the Weyl chambers of type $B_n$, we prepare some notations for a given subset $I$ of $[\pm n]=\{ \pm1, \pm2, \ldots, \pm n \}$:
   %   \begin{itemize}
%        \item $I^+=\{ i\in[n] \mid i\in I \}=I\cap \{ 1,2,\ldots,n\}$
 %       \item $I^-=\{ i\in[n] \mid -i\in I \}$
    %  \item   \blue{$I^{\pm}=\{i\in [n] \mid I\cap \{ i,-i\}\neq\emptyset\}$}
    %  \end{itemize}
%    \subsection{$K^\R_{B_n}$ and $\lambda^\R_{B_n}$}
Let $\Phi$ be a root system of type $B_n$. It consists of $2n^2$ roots
    $$
        \pm \varepsilon_i ~(1\leq i\leq n)  \quad \text{ and }  \quad \pm\varepsilon_i \pm \varepsilon_j ~(1\leq i<j \leq n),
    $$ where $\varepsilon_i$ is the $i$th standard vector of $\R^n = V$.
    One can see that the lattice $\Lambda$ consists of all integral vectors in $\R^n$.
    We note that a line containing a ray of $\Sigma_{B_n}$ is the intersection of $n-1$ hyperplanes normal to $\Delta \setminus \{ \alpha \}$, where $\Delta$ is a set of simple roots of type $B_n$ and $\alpha \in \Delta$, and the direction of the ray is determined by $\alpha$. A set of simple roots of type $B_n$ forms
    $$
         \Delta = \{\mu_1 \varepsilon_{\sigma(1)} - \mu_2\varepsilon_{\sigma(2)}, \mu_2 \varepsilon_{\sigma(2)} - \mu_3\varepsilon_{\sigma(3)}, \ldots, \mu_{n-1} \varepsilon_{\sigma(n-1)}- \mu_n\varepsilon_{\sigma(n)}, \mu_n \varepsilon_{\sigma(n)} \},
    $$ where $\mu_j = \pm1$ and $\sigma \colon [n] \to [n]$ is a permutation.
    %We label each ray by a subset $I$ of $\{ \pm1 , \pm 2, \ldots, \pm n \}$ satisfying $\{ i , -i \} \not\subset I$ for any $i=1, \ldots, n$ as follows:
    For $\alpha \in \Delta$, there exists a unique primitive integral vector $\beta = (b_1, \ldots, b_n)$ such that $(\beta, \alpha') = 0$ for all $\alpha' \in \Delta \setminus \{\alpha\}$ and $(\beta, \alpha)>0$. We note that each component $b_j$ of $\beta$ is either $\pm 1$ or $0$.
    Then, we label the ray of $\Sigma_{B_n}$ corresponding to $\alpha \in \Delta$ by the set $I = \{ jb_j \mid j=1, \ldots, n\} \subset [\pm n] =\{ \pm1, \pm2, \ldots, \pm n \}$.
    More precisely, by putting $x_i = \mu_i \varepsilon_{\sigma(i)} - \mu_{i+1}\varepsilon_{\sigma(i+1)}$ for $i=1, \ldots, n-1$ and $x_n = \mu_n \varepsilon_{\sigma(n)}$, if $\alpha = x_i$, then $\displaystyle \beta = \sum_{k=1}^{i} \mu_k \varepsilon_{\sigma(k)}$, and, hence, the corresponding label is $\{ \mu_1\sigma(1), \ldots, \mu_i \sigma(i) \}$.
    Therefore, the vertices of $K_{B_n}$ can be labelled by the nonempty subsets $I$ of $[\pm n]$ satisfying
\begin{equation} \tag{$\ast$}\label{eq:star-condition}
    \text{if $i \in I$, then $-i \not\in I$},
\end{equation}
    and the characteristic map $\lambda_{B_n}$ is
  %  $$
   %     \lambda_{B_n}(I) = \sum_{ k \in I^+ } \varepsilon_{k} - \sum_{ k \in I^- } \varepsilon_{-k}.
   % $$ \blue{(The notations $I^+$ and $I^-$ are used once, so I suggest to delete $I^+$ and $I^-$, and denote the above by
    $$
        \lambda_{B_n}(I) = \sum_{ k \in I\cap [n] } \varepsilon_{k} - \sum_{ k \in I\setminus[n] } \varepsilon_{-k}.
    $$
            As consequence,
    $$
        \lambda^\R_{B_n}(I) = \sum_{ k \in (I \cup -I) \cap [n] } \e_{k},
    $$ where $\e_i$ is the $i$th standard vector of $\Z_2^n$.

    Furthermore, one can see that each $n$-dimensional cone $C_\Delta$ in $\Sigma_{B_n}$ corresponds to $n$ subsets $I_1, \ldots, I_n$ satisfying \eqref{eq:star-condition} such that $I_1 \subsetneq \cdots \subsetneq I_n$ and vice versa. This implies that each $(\ell-1)$-dimensional simplex of $K_{B_n}$ is labelled by a nested $\ell$ subsets of $[\pm n]$ satisfying \eqref{eq:star-condition}, that is,
    $\{I_{i_1}, \ldots, I_{i_\ell}\} \in K_{B_n}$ if and only if there is a permutation $\sigma$ on $[\ell]$ such that $I_{i_{\sigma(1)}} \subset \cdots \subset I_{i_{\sigma(\ell)}}$.

\begin{example}\label{example:B_2}
    Let us consider $\Sigma_{B_2}$.
    The corresponding toric variety $X_{B_2}$ is $\C\P^2 \sharp 5\overline{\C\P}^2$, and the corresponding real toric variety $X^\R_{B_2}$ is the connected sum of six $\R\P^2$s.
    Let us compute the Betti number of $X^\R_{B_2}$ using Theorem~\ref{formula}.
    We express $\lambda^\R_{B_2}$ by a matrix and draw the geometric realization of $K_{B_2}$ as below, respectively:
\begin{center}
\setcounter{MaxMatrixCols}{20}
    \begin{tikzpicture}[scale=0.5]
        \draw (-10,0) node {$            \begin{pmatrix}
                1 & 2 & \overline{1} & \overline{2} & 12 & 1\overline{2} & \overline{1} 2 & \overline{1}\overline{2} \\ \hline
                1 & 0 & 1 & 0  & 1  & 1  & 1     & 1 \\
                0 & 1 & 0 & 1  & 1  & 1  & 1     & 1
            \end{pmatrix} $};
        \draw (2,0)--(1.5,1.5)--(0,2)--(-1.5,1.5)--(-2,0)--(-1.5,-1.5)--(0,-2)--(1.5,-1.5)--cycle;
        \draw (2,0) [right] node {$1$};
        \draw (0,2) [above] node {$2$};
        \draw (-2,0) [left] node {$\overline{1}$};
        \draw (0,-2) [below] node {$\overline{2}$};
        \draw (1.5,1.5) [above right] node {$12$};
        \draw (-1.5,1.5) [above left] node {$\overline{1}2$};
        \draw (-1.5,-1.5) [below left] node {$\overline{1}\overline{2}$};
        \draw (1.5,-1.5) [below right] node {$1\overline{2}$};
    \end{tikzpicture}
\end{center}
    where the numbers over the horizontal lines are indicators for vertices of $K_{B_2}$.
    Then, $(K_{B_2})_{\{1\}} \simeq S^1 \setminus \{ \{2\}, \{\overline{2}\}\}$ is homotopy equivalent to $S^0$, and similarly, we have
    $(K_{B_2})_{\{2\}} \simeq S^0$ and $(K_{B_2})_{\{1,2\}} \simeq \bigvee^3 S^0$.
    Therefore, the Betti number of $X^\R_{B_2}$ is
    $$ \beta^i(X^\R_{B_2};\Q) = \left\{
                                 \begin{array}{ll}
                                   1, & \hbox{$i=0$;} \\
                                   5, & \hbox{$i=1$;} \\
                                   0, & \hbox{otherwise.}
                                 \end{array}
                               \right.
    $$
\end{example}

    From now on, let us compute the $\Q$-Betti number of $X^\R_{B_n}$.
    By Theorem~\ref{formula}, we have to consider $(K_{B_n})_S$ for all subsets $S \subset [n]$.
    Given a subset $S \subset [n]$, then $(K_{B_n})_S$ is the restriction of $K_{B_n}$ by $\{ I \in V(K_{B_n}) \mid | S \cap I^\pm | \text{ is odd} \}$, where $V(K)$ is the vertex set of a simplicial complex $K$.

    Now let us consider the case where $S = [n]$. Define $K_{B_n}^{odd}$ as
    $$
        K_{B_n}^{odd} := (K_{B_n})_{[n]} = \{ \sigma \in K_{B_n} \mid \text{$\sigma$ consists of $I$ such that $| I |$ is odd} \}.
    $$

    We define the poset $S_{B_n}^{odd}$ whose vertices are the vertices of $K_{B_n}^{odd}$ and the partial order is given by inclusion, and define another poset $\tilde{S}_{B_n}^{odd} := S_{B_n}^{odd} \cup \{ \emptyset, [\pm n] \}$ with inclusion.
    Note that the order complex of $S_{B_n}^{odd}$ is $K_{B_n}^{odd}$, and hence, $\tilde{\chi}(K_{B_n}^{odd}) = \mu(\emptyset, [\pm n] )$ where $\mu$ is the M\"obius function of $\tilde{S}_{B_n}^{odd}$ (see Section~3 of \cite{Stanley1997} for details), that is,
  $$
    \mu (\rho,\tau) = \left\{
                  \begin{array}{ll}
                    1, & \hbox{if $ \rho=\tau$;} \\
                    - \sum_{\rho \leq \sigma < \tau} \mu (\rho,\sigma), & \hbox{if $\rho \subset \tau$ in $\tilde{S}_{B_n}^{odd}$.}
                  \end{array}
                \right.
  $$

  \begin{lemma} \label{lemma:comp_of_mobius}
    The absolute value of $\mu(\emptyset, [\pm n] )$ is $b_{ n}$.
    More precisely,
    \[
        \mu(\emptyset, [\pm n] ) = \left\{
                                     \begin{array}{ll}
                                       b_n, & \hbox{if $n \equiv 0,1 ~(\text{mod } 4)$;} \\
                                       -b_n, & \hbox{if $n \equiv 2,3 ~(\text{mod } 4)$.}
                                     \end{array}
                                   \right.
    \]
  \end{lemma}
  \begin{proof}

    In this proof, we use $i$ to denote the imaginary unit such that $i^2 = -1$.

    For a vertex $I$ in $\tilde{S}_{B_n}^{odd}$ such that $I$ is neither $\emptyset$ nor $[\pm n]$, put $|I| = 2k+1$.
    Note that the M\"obius  function $\mu(\emptyset, I)$ depends only on $|I|$ and $\mu(\emptyset, I) = (-1)^{k+1} a_{2k+1}$ (see the proof of Theorem~2.9 of \cite{Choi-Park2015}).
    Hence, we have
    \begin{align*}
          -\mu(\emptyset, [\pm n]) &= 1+ \sum_{k=0}^{\lfloor \frac{n-1}{2} \rfloor}  (-1)^{k+1} a_{2k+1} 2^{2k+1} \binom{n}{2k+1} \\
           &= 1 + i\sum_{k=0}^\infty a_{2k+1} (2i)^{2k+1}\binom{n}{2k+1}.
    \end{align*}

    Recall that the exponential generating functions of $a_n$ and $b_n$ are
    \[
        \sum_{n=0}^\infty a_n \frac{x^n}{n!} = \sec x + \tan x
    \]
    and
    \[
        B(x):=\sum_{n=0}^\infty b_n \frac{x^n}{n!} = \frac1{\cos x - \sin x} = (\cos x + \sin x)\sec 2x
    \]
    respectively. Since $e^x = \cos (-ix) + i \sin (-ix)$, we have
    \begin{align*}
        M(x) := -\sum_{n=0}^\infty \mu(\emptyset, [\pm n]) \frac{x^n}{n!} &= e^x (1 + i\tan(2ix))
        \\& = e^x \frac{ \cos (2ix) + i \sin (2ix)}{\cos(2ix)}
        \\& =(\cos(ix) + i \sin (ix)) \sec (2ix).
    \end{align*}
    Therefore, $M(ix) = (\cos x - i \sin x)\sec 2x$. Since the exponential generating function of $\sec x$ has only even degree terms, $\cos x \sec 2x$ contributes the even degree terms of $M(ix)$ and $\sin x \sec 2x$ contributes the odd degree terms of $-iM(ix)$. Therefore, the lemma immediately follows from that the odd degree term of $B(x)$ is equal to that of $M(ix)$ and the even degree term of $B(x)$ is equal to that of $-iM(ix)$.
  \end{proof}

We will use the following well-known lemma in \cite{Bjorner-Waches1996}. {This can be regarded as an alternative definition of shellability. Recall that a simplicial complex is called shellable if it admits a shelling.}
\begin{lemma}\cite[Lemma 2.3]{Bjorner-Waches1996}\label{lem:shellable}
An order $\mathcal{F}_1$, $\mathcal{F}_2$, $\ldots$, $\mathcal{F}_t$ of the facets of a simplicial complex is a shelling if and only if for every $i$ and $k$ with $1\le i<k\le t$ there is a $j$ with $1\le j<k$ such that $\mathcal{F}_i\cap \mathcal{F}_k \subseteq \mathcal{F}_j \cap \mathcal{F}_k$ and $|\mathcal{F}_j\cap \mathcal{F}_k|= |\mathcal{F}_k|-1$.
\end{lemma}

\begin{lemma}\label{lem:B:shellable}
For any integer $n$, $K_{B_n}^{odd}$ is shellable.
\end{lemma}
\begin{proof}
Note that since $K_{B_n}$ bounds for a convex polytope, it is shellable.
Choose a shelling $\sigma \colon  F_1,\ldots,F_t$ of $K_{B_n}$.
For each $m\in [t]$, let $F'_m$ be the face obtained from $F_m$ by deleting all vertices of $F_m$ corresponding to even subsets of $[\pm n]$.
Note that for any $m\in[t]$, $F'_m$ is a facet of $K_{B_n}^{odd}$.
Then consider an ordering $\sigma': F'_1,\ldots,F'_{t}$ of the facets of $K_{B_n}^{odd}$, and then we delete $F'_m$  whenever $F'_m=F'_\ell$ for some $\ell$ such that $\ell<m$. Let $\sigma^* \colon F^*_{1}, F^*_{2}, \ldots, F^*_{s}$ be the resulting ordering, that is, the ordering obtained from $\sigma'$ by dropping all facets of $K_{B_n}^{odd}$  not firstly appeared in $\sigma'$.
Clearly, $\sigma^*$ is an ordering of the facets of $K_{B_n}^{odd}$.
We will show that $\sigma^*$ is a shelling of $K_{B_n}^{odd}$.
By Lemma~\ref{lem:shellable}, it is enough to show that, for every $i$ and $k$ with $1 \leq i < k \leq s$, there is $j$ with $1 \leq j < k$ such that
\begin{enumerate}
  \item \label{eq:cond_1} $F_i^* \cap F_k^* \subseteq F_j^* \cap F_k^*$, and
  \item \label{eq:cond_2} $|F_j^* \cap F_k^*| = |F_k^*|-1$.
\end{enumerate}
For each $m\in [s]$, let $d_m\in [t]$ be the smallest integer such that $F^*_m \subset F_{d_m}$, i.e., $F_{d_m}$ is the first facet in $\sigma$ containing $F^*_{m}$.
Note that for all $\ell, m \in [s]$,
\begin{eqnarray}\label{ij}
d_\ell< d_m \text{ if and only if } \ell<m.
\end{eqnarray}
Take $i$ and $k$ with $1\le i<k\le s$.
Then $F^*_i \subset F_{d_i}$ and $F^*_{k}\subset F_{d_k}$.
Since $d_i <d_k$ by \eqref{ij}, by considering two facets $F_{d_i}$ and $F_{d_k}$ of $K_{B_n}$ together with Lemma~\ref{lem:shellable}, there is  $J$ with $1\le J <d_k$ such that
$F_{d_i}\cap F_{d_k} \subseteq F_{J} \cap F_{d_k}$ and $|F_{J}\cap F_{d_k}|= |F_{d_k}|-1$.
Then we consider a facet $F'_{J}$ of $K_{B_n}^{odd}$.
Let $j$ be the smallest integer such that $F^*_{j}=F'_{J}$.
Then $d_{j}\le J$ by definition, and so we have $d_{j}\le J < d_k$.
Thus $j<k$ by \eqref{ij}, and it indeed satisfies the conditions \eqref{eq:cond_1} and \eqref{eq:cond_2} as follows.

Let $V$ be the set of vertices of $K_{B_n}^{odd}$.
Note that $F^*_{i}\cap F^*_{k}=F_{d_i}\cap F_{d_k}\cap V$ and  $F^*_{j} \cap F^*_{k}=F_{J}\cap F_{d_k}\cap V$.
Therefore \eqref{eq:cond_1} follows from the fact that $F_{d_i}\cap F_{d_k} \subseteq F_{J} \cap F_{d_k}$.
Moreover, since $F^*_{j}\cap F^*_{k}=F_{J}\cap F_{d_k}\cap V$  and $|F_{J}\cap F_{d_k}|= |F_{d_k}|-1$, we have
$|F^*_{j}\cap F^*_{k}|\ge |F^*_{k}|-1$.
Since $j \neq k$ implies that $F^*_{j}\neq F^*_{k}$, \eqref{eq:cond_2} is proved.
\end{proof}

Note that $K_{B_n}^{odd}$ is homotopy equivalent to a wedge of uniform spheres $S^d$ as it is shellable.
One can easily see that the dimension of the sphere is $d = \lfloor \frac{n-1}{2} \rfloor$ by observing the dimension of the facets.
Since the absolute value of the reduced Euler characteristic of $K_{B_n}^{odd}$ is $b_n$ by Lemma~\ref{lemma:comp_of_mobius}, we conclude that $K_{B_n}^{odd} \simeq \bigvee^{b_n} S^{\lfloor \frac{n-1}{2} \rfloor}$.

\begin{remark}
Here we give an explicit shelling of $K_{B_n}^{odd}$.
We define an ordering $\prec$ on $[\pm n]$,
\[1\prec 2\prec \cdots \prec n \prec -1\prec \cdots \prec -n,\]
(just fix an ordering so that the positive integers proceed to the negative integers)
and we define an order lexicographically induced by $\prec$ on the set of all maximal chains of $\tilde{S}_{B_n}^{odd}$ (comparing the smaller element).
We also denote by the same symbol $\prec$ the order on the set of all maximal chains.
More precisely,
for  two maximal chains $\sigma$ and $\sigma'$ such that
\begin{eqnarray*}
&&\sigma: \emptyset=I_0\subsetneq I_1\subsetneq I_2\subsetneq \cdots\subsetneq I_r\subsetneq I_{r+1}=[\pm n]\\
&&\sigma': \emptyset=I'_0\subsetneq I'_1\subsetneq I'_2\subsetneq \cdots\subsetneq I'_r\subsetneq I'_{r+1}=[\pm n],
\end{eqnarray*}
we say $\sigma' \prec \sigma$ if there exists $1 \le i\le r$ such that ${I'_{i}} <_{lexi} {{I}_i}$ (comparing lexicographically under the ordering $\prec$ on $[\pm n]$)
and  ${{I'}_{j}}= {{I}_j} $  for any $j<i$.
Then it can be shown that this ordering on maximal chains gives a shelling of $\tilde{S}_{B_n}^{odd}$.
\end{remark}

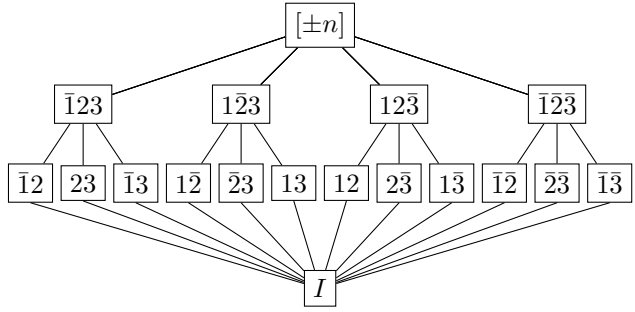
\begin{figure}
\centering
\begin{tikzpicture}[scale=.7]
	\node [draw] (-12) at (-6,0) {$\bar{1}2$};
	\node [draw] (23) at (-5,0) {$23$};
	\node [draw] (-13) at (-4,0) {$\bar{1}3$};
	\node [draw] (-123) at (-5,1.5) {$\bar{1}23$};
	\node [draw] (1-2) at (-3,0) {$1\bar{2}$};
	\node [draw] (-23) at (-2,0) {$\bar{2}3$};
	\node [draw] (13) at (-1,0) {$13$};
	\node [draw] (1-23) at (-2,1.5) {$1\bar{2}3$};
	\node [draw] (12) at (0,0) {$12$};
	\node [draw] (2-3) at (1,0) {$2\bar{3}$};
	\node [draw] (1-3) at (2,0) {$1\bar{3}$};
	\node [draw] (12-3) at (1,1.5) {$12\bar{3}$};
	\node [draw] (-1-2) at (3,0) {$\bar{1}\bar{2}$};
	\node [draw] (-2-3) at (4,0) {$\bar{2}\bar{3}$};
	\node [draw] (-1-3) at (5,0) {$\bar{1}\bar{3}$};
	\node [draw] (-1-2-3) at (4,1.5) {$\bar{1}\bar{2}\bar{3}$};
	\node [draw] (I) at (-0.5,-2) {$I$};
	\node [draw] (n) at (-0.5,3) {$[\pm n]$};
	\draw (I) -- (12.south);
\draw (12)-- (12-3)--(n);
	\draw (I) --(-12.south);
\draw  (-12)-- (-123)--(n);
	\draw (I) -- (1-2.south);
\draw (1-2) -- (1-23)--(n);
	\draw (I) --(-1-2.south);
\draw (-1-2) -- (-1-2-3)--(n);
	\draw (I) -- (13.south);
\draw (13) -- (1-23)--(n);
	\draw (I) --(-13.south);
\draw (-13) -- (-123)--(n);
	\draw (I) -- (1-3.south);
\draw (1-3) -- (12-3)--(n);
	\draw (I) --(-1-3.south);
\draw (-1-3) -- (-1-2-3)--(n);
	\draw (I) -- (23.south);
\draw (23) -- (-123)--(n);
	\draw (I) --(-23.south);
\draw (-23) -- (1-23)--(n);
	\draw (I) -- (2-3.south);
\draw (2-3) -- (12-3)--(n);
	\draw (I) --(-2-3.south);
\draw (-2-3) -- (-1-2-3)--(n);
\end{tikzpicture}
\caption{The interval $[I, [\pm n]]$ of the poset $\tilde{S}_{C_n}^{odd}$.
Here, $\bar{i}$ means $-i$ and the label in each box means a vertex in $K_{C_n}^{odd}$ obtained by the union of $I$ and the elements in the label.
For example, $\bar{1}2$ and $\bar{2}{3}$ mean $I\cup\{-1,2\}$ and $I\cup\{-2,3\}$, respectively.}  \label{Fig:C}
\end{figure}

\begin{remark}
The proof of Lemma~\ref{lem:B:shellable} is not extended naturally to the case for Weyl chambers of types $C_n$ and $D_n$. One can check that
%We can define $K_{C_n}$ similarly and then
$K_{C_n}^{odd}$ is a set of faces $\sigma$ in $K_{C_n}$ such that $\sigma$ consists of $I$ satisfying one of (1)$\sim$(4):
\begin{itemize}
\item[(1)] $n\not\in I^{\pm}$, and $|I|$ is odd;
\item[(2)] $n\in I^{\pm}$, $|I|\neq n$, $n-|I|$ is odd;
\item[(3)]$n\in I$, $|I|=n$,  $|I\setminus[n]|$ is even;
\item[(4)]$-n\in I$, $|I|=n$,  $|I\cap[n]|$ is even,
\end{itemize}
where $I^{\pm}$ denotes $(I \cup -I) \cap [n]$.
In general, $K_{C_n}^{odd}$ is not shellable.
For an illustration, consider $\tilde{S}_{C_n}^{odd}$ when $n$ is an even integer such that $n\ge 4$.
Let $I=\{4,5,\ldots,n\}$ and consider the interval $[I, [\pm n]]$ of the poset $\tilde{S}_{C_n}^{odd}$, see Figure~\ref{Fig:C}.
It is easy to see that the interval $[I, [\pm n]]$ is not shellable.
Since $\tilde{S}_{C_n}^{odd}$ has an non-shellable interval, it is not shellable.
It can be similarly shown that $K_{D_n}^{odd}$ is not shellable.
\end{remark}

Now, let us return to the case where $S \neq [n]$. If $S$ is an empty set, so is $(K_{B_n})_S$.

\begin{lemma}[Lemma~5.2 of \cite{Choi-Park2015}] \label{lem:st}
        Let $I$ be a vertex of a simplicial complex  {$K$} and suppose that the link of $I$, $\Lk I$, is contractible.
        Then {$K$} is homotopy equivalent to the complex ${K\setminus\St I}$, where $\St I$ is the star of $I$.
\end{lemma}

\begin{lemma}\label{lem:deletion_B}
For a positive integer $n\ge 3$,
for $S\subset [n]$, $(K_{B_n})_S$ is homotopy equivalent to $(K_{B_n})'_{S}$,
where $(K_{B_n})'_{S}$ is obtained  from $(K_{B_n})_S$ by deleting vertices $I$ in $(K_{B_n})_S$ such that  $I^{\pm}\not\subset S$.
\end{lemma}

\begin{proof}
For simplicity, we let $K=(K_{B_n})_S$ and  $K'=(K_{B_n})'_{S}$.
  We will show that we can eliminate stars of vertices in $K\setminus K'$, one by one, from $K$ to $K'$, without changing the homotopy type. First,
for any vertex $I$ of $K$, $I\cap S\neq\emptyset$. In addition, two  vertices $I$ and $J$ meet in $K$ if and only if $I\subset J$ or $J\subset I$.

Let $I$ be a vertex of $K\setminus K'$ such that $|I^{\pm}\cap S|=1$, say $I^{\pm}\cap S=\{x\}$.
Let $J$ be a vertex in $K$ such that $J^{\pm}=\{x\}$ and $J\subsetneq I$.
Take any  $L\in \Lk I$.
If $I\subset L$, then $J\subset L$, and so $L$ meets $J$.
Suppose that $L\subset I$. Then $L^{\pm}\cap S$ is a subset of $I^{\pm}\cap S$.
Since $L$ is a vertex of $K$, $L^{\pm}\cap S \neq \emptyset$.
Therefore $L^{\pm}\cap S=I^{\pm}\cap S=\{x\}=J^{\pm}$ and so  $J^{\pm} \subset L^{\pm}$.
Since $J^{\pm} \subset L^{\pm} \subset I^{\pm}$, $J\subset I$ and $L\subset I$, it follows that
$J\subset L$, and so $L$ meets $J$.

Hence, $\Lk I$ is contractible, and so $K$ is homotopy equivalent to $K\setminus\St  I$ by Lemma~\ref{lem:st}. By redefining $K:=K\setminus\St  I$ and repeating the argument, we can conclude that the star of any vertex $I$ with $|I^{\pm}\cap S|=1$ can be eliminated.

Inductively, assume that we could eliminate all vertices $I\in K \setminus K'$ such that $|I^{\pm}\cap S|<j$, and let
 $K^*$ be the simplicial complex obtained by deleting stars of all those vertices, where $j\ge 2$.
Take a {smallest} vertex $I\in K^*\setminus K'$ such that $|I^{\pm}\cap S|=j$.
Let $J$ be a vertex in $K^*$ such that $I^{\pm}\cap S=J^{\pm}$ and $J\subset I$.
(Note that $J\in K'$ and so $J$ is in $K^*$ and $I\neq J$.)
Take any  $L\in \Lk I$ in $K^*$.
If $I\subset L$, then $J\subset L$, and so $L$ meets $J$.
Suppose that  $L\subset I$. Then $L^{\pm}\cap S$ is a subset of $I^{\pm}\cap S=J^{\pm}$.
%First we will show that $L\not\in K'$.
%Suppose that $L\in K'$.
%Then $L^{\pm}\subset T$.
%Since  $L\subset I$, we have $L\subset I^{\pm}\cap S =J^{pm}$.
%Since
%$L^{\pm} \subset J^{\pm} \subset I^{\pm}$, $J\subset I$, $L\subset I$.
%it follows that $L\subset J$, and so $L$ meets $J$.
%Thus $L\not\in K'$.
If $|L^{\pm}\cap S|<j$, then such $L$ {should have already been deleted} by our induction hypothesis.
Thus $|L^{\pm}\cap S|=j$ and so  $L^{\pm}\cap S=J^\pm$.
Therefore $J^{\pm}\subset L^{\pm}$.
Since $J\subset I$ and $L\subset I$, we have $J\subset L$, and so $L$ meets $J$.

Hence, $\Lk I$ is contractible, and so $K^*$ is homotopy equivalent to $K^*\setminus\St I$ by Lemma~\ref{lem:st}.
        By redefining $K^*:=K^*\setminus\St I$ and repeating the argument,
        we can conclude that the star of any vertex $I$ with $|I^{\pm}\cap S|=j$ can be eliminated in increasing order of the size $|I^{\pm} \cap S|$.
\end{proof}

\begin{lemma} \label{lemma:K_B_S}
Let $r = |S|$. Then,  $ (K_{B_n})_S$ is homotopy equivalent to $K_{B_r}^{odd}$.
\end{lemma}
\begin{proof}
By Lemma~\ref{lem:deletion_B}, it clearly follows.
\end{proof}

    \begin{theorem}
     The $i$th $\Q$-Betti number $\beta^i (X^\R_{B_n};\Q)$ of $X^\R_{B_n}$  is
     $$
        \beta^i (X^\R_{B_n};\Q) = \binom{n}{2i}b_{2i} + \binom{n}{2i-1}b_{2i-1}.
     $$
    Furthermore, their integral cohomologies of $X^\R_{B_n}$ are $p$-torsion free for all odd primes $p$.
    \end{theorem}

\begin{proof}
    Let $S \subset [n]$ and assume that $|S|=r$. By Lemma~\ref{lemma:K_B_S}, $(K_{B_n})_S \cong K_{B_r}^{odd}$.
    We recall that $(K_{B_r})_S \cong \bigvee^{b_r} S^{\lfloor \frac{r-1}{2} \rfloor}$. Hence, the topology type of $(K_{B_n})_S$ only depends on the {cardinality} of $S$. For a fixed $i$ and a field $\k$ with characteristic is not equal to $2$, by Theorem~\ref{formula},
$$
    \beta^i (X^\R_{B_n};\k) = \sum_S \beta^{i-1} \tilde{\beta}((K_{B_n})_S;\k) = \sum_{r=0}^n \delta_{i-1, \lfloor \frac{r-1}{2} \rfloor} \binom{n}{r}  b_r,
$$ where $\delta_{i,j}=1$ if $i=j$ and 0 otherwise. It proves the theorem.
\end{proof}

\bigskip

\section*{Acknowledgement}
%\blue{The authors thank the anonymous referee for his/her valuable comments and suggestions.}
The authors thank to Professor Soojin Cho for helpful discussions, and Professor Jang Soo Kim for suggesting nice proof of Lemma~\ref{lemma:comp_of_mobius}. They are also thankful to the anonymous referee for the thorough reading and kind comments. The first named author was supported by Basic Science Research Program through the National Research Foundation of Korea(NRF) funded by the Ministry of Science, ICT \& Future Planning (NRF-2012R1A1A2044990). 
The second named author was supported by Basic Science Research Program through the National Research Foundation of Korea (NRF) funded by the Ministry of  Science, ICT \& Future Planning (NRF-2015R1C1A1A01053495).

\bibliographystyle{amsplain}
\providecommand{\MR}{\relax\ifhmode\unskip\space\fi MR }
\providecommand{\MRhref}[2]{%
  \href{http://www.ams.org/mathscinet-getitem?mr=#1}{#2}
}
\providecommand{\href}[2]{#2}

\end{document}